\documentclass[11pt,leqno]{amsart}
\usepackage{amsfonts,amsthm,amsmath}

\usepackage[pdftex]{graphicx}
\usepackage{amssymb}
\usepackage{color}
\theoremstyle{plain}
\newtheorem{thm}{Theorem}[section]

\newtheorem{conj}{Conjecture}[section]

\newtheorem{fact}{Fact}[section]

\theoremstyle{definition}
\newtheorem{df}{Definition}[section]
\newtheorem{rem}{Remark}[section]
\newtheorem{ex}{Example}[section]
\newtheorem{prob}{Problem}[section]

\newcommand{\FF}{\mathbb{F}}
\newcommand{\ZZ}{\mathbb{Z}}

\newcommand{\CC}{\mathbb{C}}

\newcommand{\Z}{\mathbb{Z}}
\newcommand{\NN}{\mathbb{N}}

\newcommand{\R}{\mathbb{R}}

\DeclareMathOperator{\Aut}{Aut}
\DeclareMathOperator{\wt}{wt}

\begin{document}


\title[Tutte polynomial, complete invariant, and theta series]
{Tutte polynomial, complete invariant, and theta series}

\author[Kume]{Misaki Kume}
\address{Faculty of Education, University of the Ryukyus, Okinawa  
903--0213, Japan}
\author[Miezaki]{Tsuyoshi Miezaki*}
\thanks{*Corresponding author}
\address{Faculty of Education, University of the Ryukyus, Okinawa  
903--0213, Japan}
\email{miezaki@edu.u-ryukyu.ac.jp}
\author[Sakuma]{Tadashi Sakuma}
\address{Faculty of Science, 
Yamagata University,  
Yamagata 990--8560, Japan}
\email{sakuma@sci.kj.yamagata-u.ac.jp}
\author[Shinohara]{Hidehiro Shinohara}
\address{Institute for Excellence in Higher Education, 
Tohoku University,  
Miyagi 980--8576, Japan}
\email{witchy.shino@gmail.com}

\date{July 31, 2020}



\date{}
\maketitle

\begin{abstract}
In this study, 
we present two results that relate 
Tutte polynomials.
First, 
we provide new and complete polynomial invariants for 
graphs. 
We note that 
the number of variables of our polynomials is one. 
Second, 
let $L_1$ and $L_2$ be two non-isomorphic lattices. 
We state that $L_1$ and $L_2$ are theta series equivalent 
if those theta series are the same. 
The problem of identifying theta series equivalent lattices 
is discussed in Prof.~Conway's book The Sensual (Quadratic) Form 
with the title ``Can You Hear the Shape of a Lattice?''
In this study, 
we present a method to find theta series equivalent lattices 
using matroids and their Tutte polynomials. 
\end{abstract}

\section{Introduction}
In this study, 
we provide two results that relate the 
Tutte polynomials.

The first result is as follows: 
In \cite{HJ}, de la Harpe and Jones introduced 
graph-invariant polynomials. 
We refer to these polynomials as $n$-state polynomials. 

The main focus of this paper is on 
non-directed graphs. 
Let $G$ be a graph and 
$X$ be a finite set with $n$ elements. 
$W_n=(x_{ij})$ is an $|X|\times |X|$ symmetric matrix 
indexed by the elements of $X$.
We assume that $X=\{1,2,\ldots,n\}$ and 
$x_{ij}$ are variables.
Moreover, let $\sigma:V(G)\rightarrow X$ be a state function. 
Subsequently, the $n$-state polynomials with $n(n+1)/2$ variables 
are defined as follows: 
\begin{df}[\cite{HJ}]
\[
Z_{W_n}(G)=\sum_{\sigma:{\rm state}}\prod_{uv\in E(G)}W_n(\sigma(u),\sigma(v)), 
\]
where $\sigma$ runs over all state functions. 
\end{df}

The special values of these polynomials are 
approximately the same as certain known polynomial invariants for 
graphs. 
For example, 
let 
\[
W_{{\rm Negami}}=
\begin{pmatrix}
x_1+y&y&\cdots &y\\
y&x_1+y&\cdots &y\\
\vdots&\vdots&\ddots &\vdots\\
y&y&\cdots &x_1+y
\end{pmatrix}.
\]
Then, $Z_{W_{\rm Negami}}(G)$ are known as Negami polynomials 
\cite{Negami}. 
In \cite{Oxley}, Oxley demonstrated that 
$Z_{W_{\rm Negami}}(G)$ is essentially equivalent to 
the Tutte polynomials \cite{T1,T2,T3}. 
Therefore, $n$-state polynomials are a generalization of the 
Tutte polynomials. 

Moreover, let 
$W_{{\rm extNegami}}=(x_{ij})$, 
\[
x_{ij}=
\left\{
\begin{array}{l}
x_i+y\ (i=j<n),\\
x_n+y\ (i=j\geq n),\\
y.
\end{array}
\right.
\]
Then, $Z_{W_{\rm extNegami}}(G)$ are known as extended Negami polynomials 
\cite{Negami-Kawagoe}. 

In this study, 
we demonstrate that 
$n$-state polynomials 
are complete invariants for graphs. 

\begin{thm}\label{thm:main}
The $n$-state polynomials 
$\{Z_{W_n}(G)\}_{n=1}^{\infty}$ are complete invariants for graphs. 
\end{thm}
According to the proof of Theorem \ref{thm:main}, 
for $n\geq 2$, 
a graph $G$ with $n$ vertices 
is determined by the 
$n$-state polynomials, which have $n(n+1)/2$ variables. 
In this study, 
we consider the following problems: 

\begin{prob}\label{prob:inv}





Let $\mathcal{G}$ be a set of certain graphs. 
Is there a polynomial invariant 
(hopefully equivalent to the one from the $n$-state polynomial)
such that 
the polynomials determine any graph $G\in \mathcal{G}$ with $n$ vertices 
and a number of variables less than $n(n+1)/2$? 

\end{prob}

The first purpose of this study is 
to introduce polynomial invariants for graphs 
and to provide an answer to Problem \ref{prob:inv}.
The number of variables of our polynomials is one. 
Furthermore, for any given graph $G$, if the degree of our polynomial 
is sufficiently large, then the set of terms of our polynomial is 
one-to-one corresponding to the set of terms of the $n$-state 
polynomial with the same degree. 
We refer to such polynomials as pseudo $n$-state polynomials, 
denoted by $Z_{\widetilde{W}_n}(G)$, which are 
explained in the following. 

For $\ell\in\NN$, we denote the $\ell$-th prime number as $P(\ell)$. 
We set the functions $a(i)$ on $\NN$ and $b(i,j)$ on $\NN^2$ such that 
\[
\displaystyle  a(i,j) := 
\left(\frac{i(i + 1)}{2}+(j-i)\right). 
\]
Let $\widetilde{W}_n=(\widetilde{W}_n(i,j))$
be the $n\times n$ symmetric matrix such that, 
for $i\geq j$, 
\[
\widetilde{W}_n(i,j) := 
\left\{
\begin{array}{l}
 P(n^{n a(i,i)}) 
x^{P(n^{n a(i,i)})}\ (\mbox{if }i=j),\\
 P(n^{n a(i,j)}) 
x^{P(n^{n a(i,j)})}\ (\mbox{if }i> j).
\end{array}
\right.
\]
Next, we present the definition of the pseudo $n$-state polynomials: 
\begin{df}

Let $\sigma:V(G)\rightarrow X$ be a state function. 
Then, the pseudo $n$-state polynomials 
are defined as follows: 
\[
Z_{\widetilde{W}_n}(G)
=\sum_{\sigma:{\rm state}}
\prod_{uv\in E(G)}\widetilde{W}_n(\sigma(u),\sigma(v)), 
\]
where $\sigma$ runs over all state functions. 
\end{df}
We note that 
the number of variables of our polynomials is one. 
For example, 
for $n=2$,
\[
\widetilde{W}_2=
\begin{pmatrix}
7 x^{7}&53 x^{53}\\
53 x^{53}&311 x^{311}
\end{pmatrix}
\]
and for $n=3$,
\[
\widetilde{W}_3
=
\begin{pmatrix}
 103 x^{103} & 5519 x^{5519} & 7867547 x^{7867547} \\
 5519 x^{5519} & 220861 x^{220861} & 262960091 x^{262960091} \\
 7867547 x^{7867547} & 262960091 x^{262960091} & 8448283757 x^{8448283757} 
\end{pmatrix}.
\]
Then, the 
pseudo $2$-state polynomial for the complete graph $K_2$ is 
\[
Z_{\widetilde{W}_2}(K_2)=
311 x^{311}+106 x^{53}+7 x^7
\]
and 
the pseudo $3$-state polynomial for the complete graph $K_3$ is 
\begin{align*}
Z_{\widetilde{W}_3}(K_3)&=
1752546015417169494746495151 x^{8974203939}\\
&+1568803100908626481902639 x^{8464018851}\\
&+602983567540694711837927399093 x^{25344851271}\\
&+45816295551192560609823 x^{526141043}\\
&+68507927876961253578 x^{270833157}\\
&+19126573401337581 x^{15735197}\\
&+10773507110137381 x^{662583}+20181854789463 x^{231899}\\
&+9411942549 x^{11141}+1092727 x^{309}. 
\end{align*}
The first main result in this paper is as follows: 
\begin{thm}\label{thm:main2}
The pseudo $n$-state polynomial 
$\{Z_{\widetilde{W}_n}\}_{n=1}^\infty$ is a complete invariant 
for graphs. 
%
\end{thm}

In the following, 
we present the second result of the study. 
Let $L_1$ and $L_2$ be two non-isomorphic lattices. 
We state that $L_1$ and $L_2$ are theta series equivalent 
if those theta series are the same:
\[
\theta_{L_1}(q)=\theta_{L_2}(q). 
\]
The problem of determining theta series equivalent lattices 
is discussed in Prof.~Conway's book \cite{Conway} 
under the title ``Can You Hear the Shape of a Lattice?''
\begin{prob}[{\cite[Can You Hear the Shape of a Lattice?]{Conway}}]\label{prob:theta}

Finding theta series equivalent lattices. 

\end{prob}
For example, it is well known, as per the example of Milnor, that 
$E_8^2$ and $D_{16}^+$ are theta series equivalent 
\cite{{Conway},{SPLAG}}, and several examples are provided 
in \cite{Conway}. 

In this study, 
we present a method to find theta series equivalent lattices 
using matroids and their Tutte polynomials. 
The second main result is as follows:
\begin{thm}\label{thm:main3}
Let $d\in 
\{24, 27, 30, 33, 35, 36, 38, 39, 41, 42\}\cup \{i\in \ZZ\mid i\geq 44\}$. 
Then, there exist non-isomorphic lattices of rank $4d$
with the same theta series. 
\end{thm}

\begin{rem}
Prior to concluding this section, 
we provide a remark. 
The relationships and analogies among 
codes, lattices, and vertex operator algebras 
are well known in the algebraic combinatorics community. 
However, 
the proof of Theorem \ref{thm:main3} 
uses the relationships among 
matroids, codes, and lattices. 
We remark that there may be a rich theory 
behind the matroids and the three objects 
codes, lattices, and vertex operator algebras. 
\end{rem}


The remainder of this paper is organized as follows: 
In Section \ref{sec:pre}, 
we summarize basic facts of matroids, codes, and lattices. 
In Section \ref{sec:main+2},
we provide the proofs of Theorems \ref{thm:main} and~\ref{thm:main2}. 
In Section \ref{sec:main3}, 
we provide the proof of Theorem \ref{thm:main3}. 
Finally, in Section~\ref{sec:rem}, we present several remarks.

All computer calculations in this study
were performed with the aid of {\sc Magma}~\cite{Magma}
and {\sc Mathematica}~\cite{Mathematica}.

\section{Preliminaries}\label{sec:pre}

In this section, 
we summarize basic facts of matroids, codes, and lattices. 

\subsection{Matroids}

Let $E$ be a set. A matroid $M$ on $E=E(M)$ is a pair 
$(E,\mathcal{I})$, where $\mathcal{I}$ is a non-empty 
family of subsets of $E$ with the following properties: 
\[
\begin{cases}
\mbox{(i)} &\mbox{ if } I\in \mathcal{I} \mbox{ and } J\subset I, \mbox{ then } 
J\in \mathcal{I}; \\
\mbox{(ii)} &\mbox{ if } I_1,I_2\in \mathcal{I} \mbox{ and }|I_1|<|I_2|, \\
&\mbox{ then there exists }
e\in I_2\setminus I_1 \\
&\mbox{ such that } I_1\cup \{e\}\in \mathcal{I}. 
\end{cases}
\]
Each element of the set $\mathcal{I}$ is known as an independent set. 
A matroid $(E, \mathcal{I})$ is isomorphic to another matroid 
$(E', \mathcal{I}')$ 
if there exists a bijection $\varphi$ from $E$ to $E'$ such that 
$\varphi(I)\in \mathcal{I}'$ holds for each member $I\in \mathcal{I}$, 
and  
$\varphi^{-1}(I')\in \mathcal{I}$ holds for each member $I'\in \mathcal{I}'$. 

It follows from the second axiom that all maximal independent sets 
in a matroid $M$ take the same cardinality, 
known as the rank of $M$. These maximal independent sets 
$\mathcal{B}(M)$ are referred to as 
the bases of $M$. 
The rank $\rho(A)$ of an arbitrary subset $A$ of $E$ is the cardinality of the largest independent set contained in $A$. 



We provide examples below. 
\begin{ex}
Let $A$ be a $k\times n$ matrix over a finite field $\FF_q$. 
This results in a matroid $M$ on the set 
\[
E={\{z\in\ZZ \mid 1\leq z\leq n\}},
\] 
in which the set $\mathcal{I}$ is independent if and only if the family of columns of $A$ with indices belonging to $\mathcal{I}$ is linearly independent. 
Such a matroid is called a vector matroid. 
\end{ex}
\begin{ex}

Let $G=(V,E)$ be a non-directed finite graph, 
where $V$ is the vertex set of $G$ and 
$E$ is the edge set of $G$. 
Let $\mathcal{I}$ be the set of all subsets $A$ of $E$, 
such that the graph $(V,A)$ is acyclic. 
Then, $(E, \mathcal{I})$ is a matroid. 
Such a matroid is known as graphic and
is denoted by $M(G)$. 

\end{ex}

The following fact is used in the proof of 
Theorem \ref{thm:main3}: 
\begin{fact}\label{fact:matroids}
The incidence matrix of $M(G)$ provides the 
vector matroid over $\FF_2$, which is isomorphic to 
$M(G)$ as matroids. 
\end{fact}

The classification of matroids is 
one of the most important problems in the theory of matroids. 
The Tutte polynomials are a tool for classifying the matroids. 
Let $M$ be a matroid on the set $E$ with a rank function $\rho$. 
The Tutte polynomial of $M$ is defined as follows \cite{{T1},{T2},{T3}}: 
\begin{align*}
T(M)&:=T(M;x,y)\\
&:=\sum_{A\subset E}(x-1)^{\rho (E)-\rho (A)}(y-1)^{|A|-\rho (A)}. 
\end{align*}

It can easily be demonstrated that $T(M;x,y)$ is a matroid invariant{.} 
{Two matroids are $T$-equivalent if their Tutte polynomials
  are equivalent.}
It is well known that there exist two inequivalent matroids, 
which are $T$-equivalent 
(for example, see \cite[p.~269]{Welsh} or 
Section \ref{sec:main3}), 
and these examples are 
key facts for the proof of Theorem \ref{thm:main3}.

\subsection{Codes}\label{sec:codes}


Let $\FF_2$ be the finite field of two elements. 
A linear code $C$ with length $n$ is a linear subspace of $\FF_{2}^{n}$. 
An inner product $({x},{y})$ on $\FF_2^n$ is given 
by
\[
(x,y)=\sum_{i=1}^nx_iy_i,
\]
where $x,y\in \FF_q^n$ with $x=(x_1,x_2,\ldots, x_n)$ and 
$y=(y_1,y_2,\ldots, y_n)$. 
The dual of the linear code $C$ is defined as follows: 
\[
C^{\perp}=\{{y}\in \FF_{2}^{n}\mid ({x},{y}) =0 \text{ for all }{x}\in C\}.
\]

A linear code $C$ is known as self-dual if $C=C^{\perp}$. 
For $x \in\FF_2^n$,
the weight $\wt(x)$ is the number of its nonzero components. 
A self-dual code $C$ is doubly even if all codewords of $C$ have a weight that is divisible by four. 

Let $C$ be a linear code of length $n$. 
The weight enumerator associated with $C$ is
\[
w_{C}(x,y)=\sum_{c\in C}x^{n-\wt(c)}y^{\wt(c)}.
\]
For example, 
let $C$ be a doubly even self-dual code. 
Then, 
\[
w_{C}(x,y)\in\CC[P_8,P_{24}], 
\]
where
$P_{8}=x^8+14x^4y^4+y^8,
P_{24}=x^4y^4(x^4-y^4)^4$ \cite{{SPLAG},{Ebeling}}. 

\subsection{Lattices}\label{sec:lattices}

A lattice in $\R^{n}$ is a subgroup $L \subset \R^{n}$ 
with the property that there exists a basis 
$\{e_{1}, \ldots, e_{n}\}$ of $\R^{n}$ 
such that $L =\Z e_{1}\oplus \cdots \oplus\Z e_{n}$.
$n$ is known as the rank of $L$. 
The dual lattice of $L$ is the lattice
\[
L^{\sharp}:=\{y\in \R^{n}\mid (y,x) \in\Z , \ \forall x\in L\}, 
\]
where $(x,y)$ is the standard inner product. 
A lattice $L$ is integral if 
$(x,y) \in\Z$ for all $x$, $y\in L$. 
An integral lattice $L$ is referred to as even if $(x,x) \in 2\Z$ for all $x\in L$. 
An integral lattice $L$ is referred to as unimodular if $L^{\sharp}=L$.

The norm of a vector $x$ is defined as $(x, x)$.
A unimodular lattice with even norms is said to be 
even. 
An $n$-dimensional even unimodular lattice exists if and only
if $n \equiv 0 \pmod 8$. 
For example, 
the unique even unimodular lattice 
of rank $8$, namely $E_8$, exists, and 
only two even unimodular lattices 
of rank $16$ exist, namely $E_8^2$ and $D_{16}^+$, 
which we mentioned as Milnor's example of 
Problem \ref{prob:theta}. 
Moreover, the unique even unimodular lattice without roots 
of rank $24$ is the Leech lattice $\Lambda_{24}$.

Let $\mathbb{H} :=\{z\in\CC\mid {\rm Im}(z) >0\}$ be the upper half-plane. 
\begin{df}
For an integral lattice in $\R^{n}$, 
the function 
on $\mathbb{H}$ defined by
\[
\theta _{L} (q):=\sum_{x\in L}
e^{\pi iz(x,x)}=\sum_{x\in L}
q^{(x,x)}
\]
is known as the theta series of $L$, 
where $q=e^{\pi iz}$. 
\end{df}
For example, we consider an even unimodular lattice $L$. 
Then, the theta series $\theta_{L}(q)$ of $L$ 
is a modular form with respect to $SL_{2}(\Z)$. 
More precisely, 
we have 
\[
\theta_{L}\in
\CC[E_4,\Delta]
\]
where 
\begin{align*}
E_4(q)&:=1+240\sum_{n=1}^{\infty}\sigma_{3}(n)q^{2n}, \\
E_6(q)&:=1-504\sum_{n=1}^{\infty}\sigma_{5}(n)q^{2n}, \\
\Delta(q)&:=\frac{E_4(q)^3-E_6(q)^2}{1728}.
\end{align*}
where $\sigma_{k-1}(n):=\sum_{d\mid n}d^{k-1}$ \cite{{SPLAG},{Ebeling}}.


\section{Proofs of Theorems \ref{thm:main} and \ref{thm:main2}}
\label{sec:main+2}

\subsection{Proof of Theorem \ref{thm:main}}\label{sec:main}

In this section, 
we prove Theorem \ref{thm:main}.
\begin{proof}[Proof of Theorem \ref{thm:main}]
We recover the graph structure of $G=(V,E)$ from the polynomials 
$\{Z_{W_n}(G)\}_{n=1}^{\infty}$. 

First, we compute the number of the vertices $|V|$. 
According to \cite[Corollary 2.2]{Negami}, it is possible to compute 
the number of vertices: 
\[
|V|=\log_2(Z_{W_2}(G;x_{11}\rightarrow 1,x_{12}\rightarrow 1,x_{22}\rightarrow 1)). 
\]
We denote the number of vertices as $n$. 

Second, we demonstrate that $Z_{W_n}$ determines the 
graph structure.
The reason is as follows: 
We seek a term 
\[
(\mbox{constant})\times
x_{i_{11}i_{12}}^{e_1}
x_{i_{21}i_{22}}^{e_2}
\cdots
x_{i_{m1}i_{m2}}^{e_m}
\]
in $Z_{W_n}(G)$ such that 
\[
\sharp\{i_{11},i_{12},i_{21},i_{22},\ldots,i_{m1},i_{m2}\} 
\]
is the maximum among all of the terms. 
It can easily be observed that 
this term is expressed by a bijective state function $\sigma$. 
If
\[
\{1,2,\ldots,n\}=\{i_{11},i_{12},i_{21},i_{22},\ldots,i_{m1},i_{m2}\}, 
\]
$G$ is a graph with $n$ vertices
\[
\{1,2,\ldots,n\},
\]
and 
$i_{j1}$ and $i_{j2}$ are adjacent with $e_j$ edges. 

If
\[
\sharp \{1,2,\ldots,n\}>\sharp \{i_{11},i_{12},i_{21},i_{22},\ldots,i_{m1},i_{m2}\}
\]
and 
assuming that 
\[
\{j_1,\ldots,j_\ell\}:=\{1,2,\ldots,n\}\setminus \{i_{11},i_{12},i_{21},i_{22},\ldots,i_{m1},i_{m2}\},
\]
$G$ is a graph with $n$ vertices
\[
\{1,2,\ldots,n\}
\]
and $\ell$ isolated vertices 
\[
\{j_1,\ldots,j_\ell\},
\]
such that 
$i_{j1}$ and $i_{j2}$ are adjacent with $e_j$ edges. 
This completes the proof of Theorem \ref{thm:main}. 
\end{proof}

\subsection{Proof of Theorem \ref{thm:main2}}\label{sec:main2}

In this section, 
we prove Theorem \ref{thm:main2}.
\begin{proof}[Proof of Theorem \ref{thm:main2}]
We recover the graph structure from the polynomials 
$\{Z_{\widetilde{W}_n}(G)\}_{n=1}^{\infty}$.

First, 
we compute the number of edges $|E|$. 
It is possible to compute 
the number of edges as follows: 
\[
|E|=\log_2 (Z_{\widetilde{W}_1}(G;x\rightarrow 1)). 
\]
We denote the number of edges of $G$ as $m$. 


We demonstrate that $Z_{\widetilde{W}_{3m}}$ determines the 
graph structure.
The reason is as follows: 
Let 
\[
Z_{\widetilde{W}_{3m}}(G)=\sum_{i=1}^{\ell_G}c(i)x^{i}. 
\]

For all $i\in \{1,\ldots,\ell_G\}$, 
we compute the prime factorization of 
$c(i)$ and seek the corresponding indices 
of $\widetilde{W}_{3m}$ (say, $I(c(i))$). 
For example, we recall that 
\[
\widetilde{W}_3
=
\begin{pmatrix}
 103 x^{103} & 5519 x^{5519} & 7867547 x^{7867547} \\
 5519 x^{5519} & 220861 x^{220861} & 262960091 x^{262960091} \\
 7867547 x^{7867547} & 262960091 x^{262960091} & 8448283757 x^{8448283757} 
\end{pmatrix}.
\]
Then, 
the pseudo $3$-state polynomial for the complete graph $K_2$ is 
\begin{align*}
Z_{\widetilde{W}_3}(K_2)&=
8448283757 x^{8448283757}+525920182 x^{262960091}\\
&+15735094 x^{7867547}+220861 x^{220861}+11038 x^{5519}+103 x^{103}. 
\end{align*}
Because of $103 =1\times 103$,
the term 
$103x^{103}$ 
yields 
\[
I(103)=\{\{1,1\}\}, 
\]
and 
because of $11038=2\times 5519$,
the term 
$11038 x^{5519}$ 
yields 
\[
I(11038)=\{\{1,2\}\}.
\]
Let 
\[
\widetilde{I}(c(i))=\{i\in \NN\mid i\in I, I\in I(c(i))
\setminus\{\{i,i\}\in \NN^2\}\}
\]
and $j$ be the index such that
$\sharp \widetilde{I}(c(j))$ is the maximum for all 
\[
\{
\widetilde{I}(c(i))\mid i\in \{1,\ldots,\ell_G\}
\}. 
\]
We set $n':=\sharp \widetilde{I}(c(j))$. 
Thereafter, we recover the edges of $G$ as follows: 
We use $G'$ to denote the subgraph of $G$ except for 
all isolated vertices. 

Clearly we have that $|V(G')| \leq 2|E(G')| \leq 2m$. 
Let $p(i,j)x^{p(i,j)}$ denote the $(i,j)$-entry of 
$\widetilde{W}_{3m}$. 
Then we have 
\[
mp(i,i) < p(i+1,1) \textrm{~and~} mp(i,j-1) < p(i,j)
\]
for every pair $(i,j)$ with $i \geq j$. 
Hence each term of $Z_{\widetilde{W}_{3m}}$ is one-to-one
corresponding to a unique state upto automorphism of $G^{\prime}$.


$G'$ is a graph with $n'$ vertices
\[
\{1,2,\ldots,n'\}
\]
and for 
\[
I(c(j))=\{\underbrace{\{i_1,i_2\},\ldots,\{i_1,i_2\}}_{e_{1,2}},\ldots\}, 
\]
$i_{1}$ and $i_{2}$ are adjacent with $e_{1,2}$ edges. 

Finally, we recover the number of isolated vertices. 
It can easily be observed that 
\[
c(j)=(3m)^{|V(G)|-|V(G^{\prime})|}\times |\Aut(G^{\prime})|. 
\]
Subsequently, we recover the number of isolated vertices: $|V(G)|-|V(G')|$. 

This completes the proof of Theorem \ref{thm:main2}. 
\end{proof}

\section{Proof of Theorem \ref{thm:main3}}\label{sec:main3}
In this section, we present the proof of Theorem \ref{thm:main3}. 
Prior to this, 
in Section \ref{sec:MC}, 
we provide a relationship between matroids and codes, 
and in Section \ref{sec:CL}, 
we present a relationship between codes and lattices. 
\subsection{Relationship between matroids and codes}\label{sec:MC}

In \cite{Greene}, 
a relationship between 
the weight enumerators of codes and 
the Tutte polynomials of matroids was presented. 

Let $M$ be a vector matroid obtained from the $k\times n$ matrix $A$. 
Then, the row space of $A$ is a code over $\FF_2$ of length $n$. 
We denote such a code as $C_M$. 
The Tutte polynomial of a vector matroid $M$ 
and the weight enumerator of $C_M$ exhibit the following relation:
\begin{thm}[\cite{Greene}]\label{thm:Greene}
Let $M$ be a vector matroid on a set $E=\{1,\ldots,n\}$ over $\FF_2$. 
Then,
\[
w_{C_M}(x_1,x_2)=
x_2^{n-\dim(C_M)}(x_1-x_2)^{\dim(C_{M})}
T\left(M; \frac{x_1+x_2}{x_1-x_2},\frac{x_1}{x_2}\right). 
\]
\end{thm}

\subsection{Relationship between codes and lattices}\label{sec:CL}

We propose a method to construct 
lattices from codes over $\FF_2$, which 
is referred to as Construction A~\cite{{BDHO},{HMV}}. 
Let $\rho$ be a map from $\ZZ^n$ to $\FF_2^n$ 
sending $(x_i)$ 
to $(x_i\pmod{2})$. 
If $C$ is a $\FF_{2}$ code of length $n$, 
we have an $n$-dimensional unimodular lattice 
\[
L_C=\frac{1}{\sqrt{2}}\{x\in \ZZ^n \mid \rho (x)\in C\}. 
\]
The following is an established fact \cite{Ebeling}:
\begin{thm}[\cite{Ebeling}]\label{thm:Ebeling}
Let $C$ be a code and $L_C$ be a lattice obtained from $C$ by 
Construction A. 
Then,
\[
w_{C}(\theta_3,\theta_2)=
\theta_{L_C}(q), 
\]
where 
\[
\theta_3=\sum_{n\in \ZZ}q^{n^2}, 
\theta_2=\sum_{n\in \ZZ+1/2}q^{n^2}.
\]
\end{thm}
We explain Milnor's example of Problem \ref{prob:theta} \cite{SPLAG}. 
There exist non-isomorphic doubly even self-dual codes of length $16$, $e_8^2$ and $d_{16}^+$ \cite{SPLAG}. 
Then, according to section \ref{sec:codes}, 
their weight enumerators are the same:
\[
w_{e_8^2}(x,y)=w_{d_{16}^+}(x,y)=P_8^2=(x^8+14x^4y^4+y^8)^2.
\]
Moreover, 
we obtain the non-isomorphic unimodular lattices of rank $16$, 
$E_8^2=L_{e_8^2}$ and $D_{16}^+=L_{d_{16}^+}$,
with the same theta series: 
\begin{align*}
\theta_{E_8}(q)&=\theta_{D_{16}^+}(q)\\
&=w_{e_8^2}(\theta_3,\theta_2)=w_{d_{16}^+}(\theta_3,\theta_2)
=(\theta_{3}^8+14\theta_{3}^4\theta_{2}^4+\theta_{2}^8)^2=E_4(q).
\end{align*}
Therefore, this example arises from the following concept:
\begin{center}
Codes $\longrightarrow$ Lattices.
\end{center}
The main idea of the proof of Theorem \ref{thm:main3} is 
that we add ``Matriods'' to the above diagram:
\begin{center}
Matroids $\longrightarrow $ Codes $\longrightarrow$ Lattices.
\end{center}
Thus, first, we have non-isomorphic graphic matroids with the 
same Tutte polynomials, 
following which we obtain non-isomorphic codes and lattices, which 
have the same theta series. 
We explain this in further detail in the following section. 
\subsection{Proof of Theorem \ref{thm:main3}}

In this section, we demonstrate the proof of Theorem \ref{thm:main3}. 

\begin{proof}[Proof of Theorem \ref{thm:main3}]


Let $G_1$ and $G_2$ be the graphs in Figure \ref{fig:graphs} \cite{BPR}. 

\begin{figure}[h]
\begin{tabular}{cc}
\hspace{-25pt}
\includegraphics[width=10cm,clip]{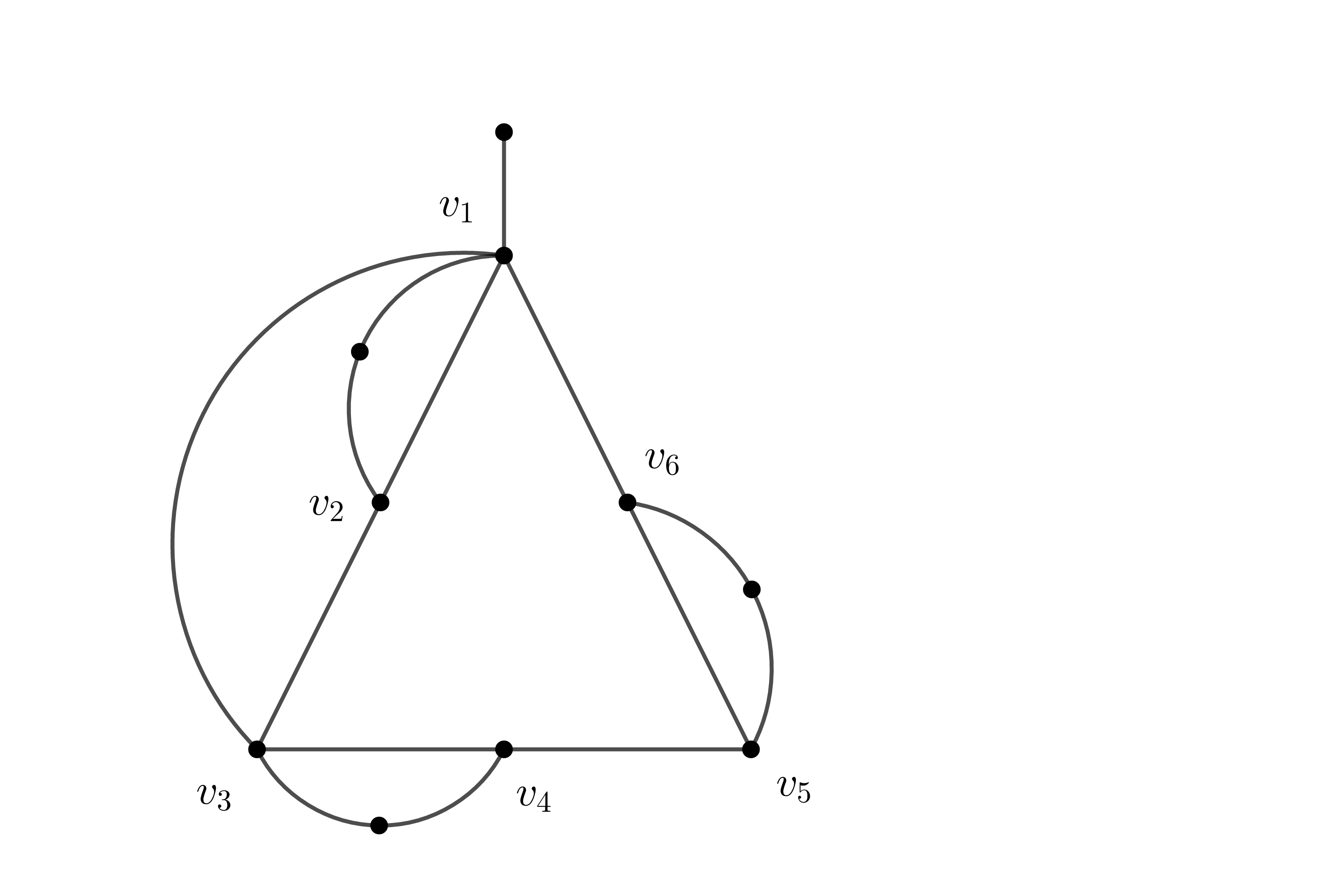}
&
\hspace{-125pt}
\includegraphics[width=10cm,clip]{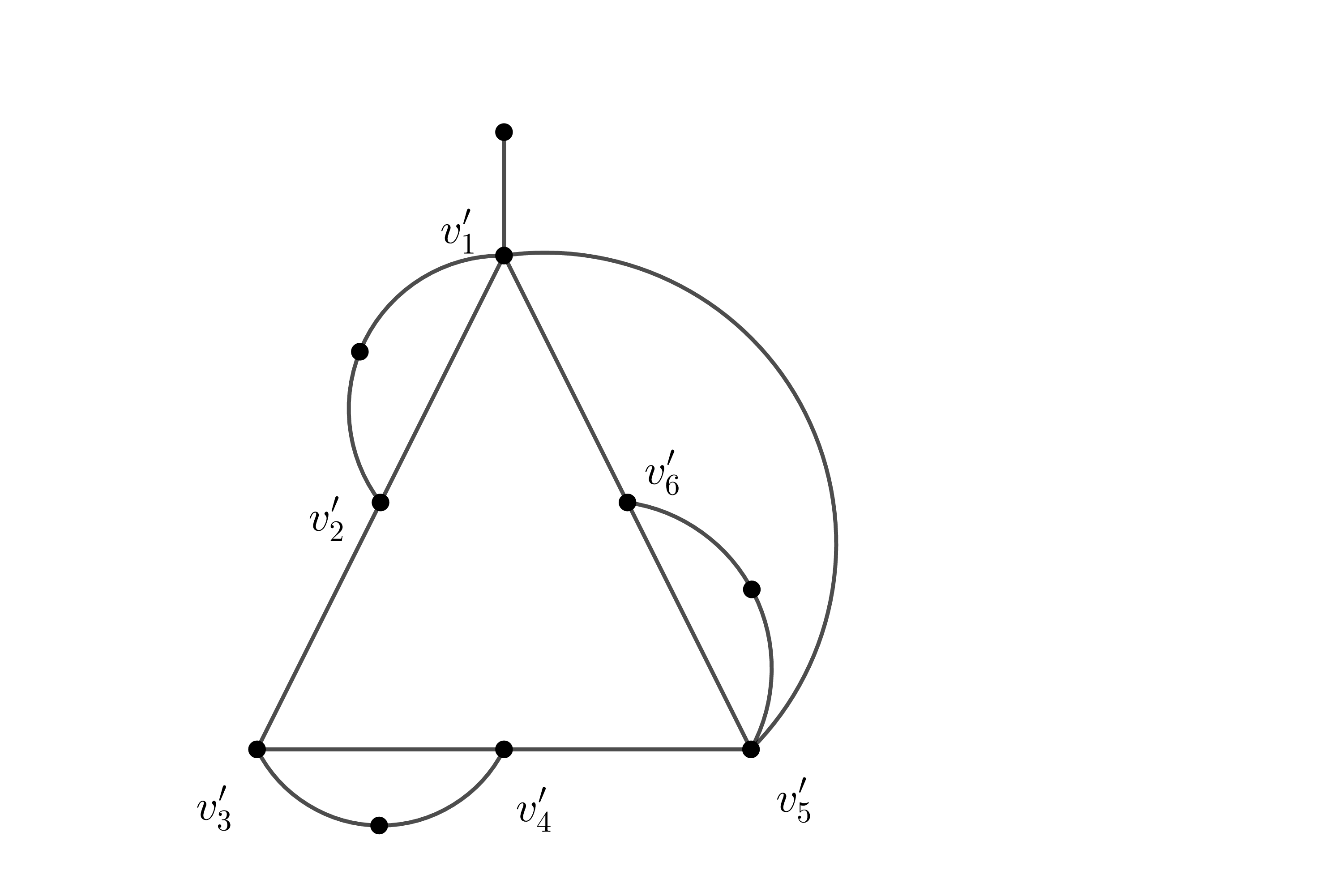}
\end{tabular}
\vspace{-20pt}
\caption{$G_1$ and $G_2$}\label{fig:graphs}
\end{figure}

Furthermore, for $i\in \{1,2\}$ and $n\in \ZZ_{\geq 0}$, 
let $G_i(n)=G_i \ast P_n$ be the joining of the graphs 
$G_i(n)$ and $P_n$, where $P_n$ indicates the path graph 
with $n$ edges. 
Note that the joining $G = G_1 \ast G_2$ of 
graphs $G_1=(V_1,E_1)$ and $G_2=(V_2,E_2)$ 
is the graph union of the 
$G_1$ and $G_2$ together with all 
edges joining $V_1$ and $V_2$. 

For $i\in\{1,2\}$ and $m\in \ZZ_{\geq 0}$, 
let $G_i(m,n)$ be the graphs 
with $m$ times edge subdivisions with respect to the edges:
\[
v_1v_2,
v_3v_4,
v_5v_6,
v'_1v'_2,
v'_3v'_4,
v'_5v'_6.
\]
We note that, for $i\in \{1,2\}$, the number of 
edges of $G_i(m,n)$ is $3m+11n+24$. 

It was demonstrated in \cite{BPR} that 
$G_1(m,n)$ and $G_2(m,n)$ are non-isomorphic graphic matroids
with the same Tutte polynomials. 
For $i\in \{1,2\}$, 
let $M_i$ be the vector matroid with respect to 
the incident matrix of $G_i(m,n)$. 
Then, according to Fact \ref{fact:matroids} and Theorem \ref{thm:Greene}, 
we can obtain the non-isomorphic codes 
$C_{M_1}$ and $C_{M_2}$ with length $3m+11n+24$ 
that have the same weight enumerator. 

Let $\varphi:\FF_2\rightarrow \FF_2^4$ be a map such that 
\[
0\mapsto 0000, 1\mapsto 1111.
\]
For $i\in \{1,2\}$, 
we define a new code with length $4(3m+11n+24)$: 
\[
\widetilde{C}_{M_i}:=
\{(\varphi(c_i))\mid (c_i)\in {C}_{M_i}\}. 
\]
We note that $\widetilde{C}_{M_i}$ is doubly even. 
According to \cite{{KKM},{Shima}} and Theorem \ref{thm:Ebeling}, 
we can obtain the non-isomorphic lattices 
$L_{\widetilde{C}_{M_1}}$ and $L_{\widetilde{C}_{M_2}}$ with the 
same theta series. 

The rank of $L_{\widetilde{C}_{M_1}}$ and $L_{\widetilde{C}_{M_2}}$ is 
$4(3m+11n+24)$. 
We remark that 
\[
3m+11n+24
\]
represents the numbers 
\[
\{24, 27, 30, 33, 35, 36, 38, 39, 41, 42\}\cup \{i\in \ZZ\mid i\geq 44\}. 
\]
This completes the proof of Theorem \ref{thm:main3}. 
\end{proof}

\section{Concluding remarks}\label{sec:rem}
We provide the following remarks: 

\begin{enumerate}
\item 
According to the proof of Theorem \ref{thm:main2}, 
the pseudo $3m$-state polynomials recover the 
graph structure with a number of edges less than or equal to $m$. 
It is natural to ask whether 
there exists a function $f(n)$ on $n\in \NN$ such that 
the polynomials with $f(n)$ variables recover the 
graph structure with a number of vertices less than 
or equal to $n$. 

\item 
Problem: 
For $n\in \NN$, 
determine the set of graphs $\mathcal{G}_n$ such that 
the pseudo $n$-state polynomials determine 
any graph $G\in \mathcal{G}_n$. 

\item 

Problem: 
In \cite{{MOSS1},{MOSS2}}, we defined a complete invariant for 
matroids. 
For a graphic matroid, determine whether or not
this invariant is a special value of $n$-state polynomials. 




\item 
In the proof of Theorem \ref{thm:main3}, 
we demonstrated that 
$L_{\widetilde{C}_{M_1}}$ and $L_{\widetilde{C}_{M_2}}$ 
are non-isomorphic with rank $4d$, where
\[
d\in \{24, 27, 30, 33, 35, 36, 38, 39, 41, 42\}\cup \{i\in \ZZ\mid i\geq 44\}. 
\]
However, for 
\[
d\in \{24, 27, 30, 33, 35\},
\]
we showed, with the aid of {\sc Magma}, that 
$L_{{C}_{M_1}}$ and $L_{{C}_{M_2}}$ 
are non-isomorphic. Therefore, 
we have the following conjecture:
\begin{conj}\label{conj:theta}
$L_{{C}_{M_1}}$ and $L_{{C}_{M_2}}$, as defined in the 
proof of Theorem \ref{thm:main3}, 
are non-isomorphic. 
\end{conj}
If Conjecture \ref{conj:theta} is true, 

we obtain examples of Problem \ref{prob:theta} for 
rank 
\[d\in 
\{24, 27, 30, 33, 35, 36, 38, 39, 41, 42\}\cup \{i\in \ZZ\mid i\geq 44\}. 
\]
\item 
The Tutte polynomials of genus $g$ were defined in \cite{MOSS1} and 
we discuss their properties in a forthcoming paper \cite{MOSS2}. 
According to Theorem \ref{thm:Greene}, 
there is a correspondence between 
the Tutte polynomials and weight enumerators. 

Thus, it is also natural to ask whether 
there is a correspondence 
between 
the Tutte polynomials and weight enumerators
of genus $g$. 

It is well known that 
the weight enumerators 
of genus $g$ yield the 
Siegel theta series. 
Therefore, 
if such a correspondence exists, 
we may obtain examples whereby the 
non-isomorphic lattices have the 
same Siegel theta series. 

\item 

Using the concept of this study 
and the following diagram:
\begin{center}
Matroids $\longrightarrow $ Codes $\longrightarrow$ Lattices $\longrightarrow$ 
Vertex operator (super) algebras, 
\end{center}
we may obtain the non-isomorphic vertex operator (super) algebras 
with the same trace function. 

\end{enumerate}

\section*{Acknowledgments}

The authors thank Prof.~Akihiro Munemasa and Prof.~Manabu Oura 
for their helpful discussions 
on this research. 
The authors would also like to thank the anonymous
reviewers for their beneficial comments on 
an earlier version of the manuscript. 
This work was supported by JSPS KAKENHI (18K03217, 18K03388).


\end{document}